\documentclass[12pt]{amsart}
\usepackage[dvips]{epsfig}
\usepackage{graphics}
\usepackage{latexsym}
\usepackage{verbatim}
\usepackage{amsmath}
\usepackage{amsthm}
\usepackage{amssymb}
\usepackage{hyperref}
\usepackage[]{hyperref}
\hypersetup{
    colorlinks=true,%
    filecolor=black,%
    linkcolor=black,%
    urlcolor=black  %
}
\usepackage [table]{xcolor}
\usepackage{multirow}
\usepackage{float}
\usepackage{tikz}
\usepackage{subfig}

\graphicspath{{images/}}	
\newtheorem{theorem}{Theorem}[section]
\newtheorem{lemma}[theorem]{Lemma}

\newtheorem{proposition}[theorem]{Proposition}

\newtheorem{remark}[theorem]{Remark}

\newcommand\CC{\hbox{C\kern -.58em {\raise .54ex \hbox{$\scriptscriptstyle |$}}
  \kern-.55em {\raise .53ex \hbox{$\scriptscriptstyle |$}} }}
\newcommand\NN{\hbox{I\kern-.2em\hbox{N}}}
\newcommand\RR{\mathbb{R}}
\newcommand\SSS{\mathbb{S}}

\newcommand\ZZ{{{\rm Z}\kern-.28em{\rm Z}}}

\newcommand\xx{ \mathbf{x} }

\newcommand\vv{ \mathbf{v} }
\newcommand\ds{ \displaystyle }

\newcommand\bB{{\mathbf B}}

\newcommand\bJ{{\mathbf J}}

\newcommand\bP{{\mathbf P}}

\newcommand\bR{{\mathbf R}}

\newcommand\bU{{\mathbf U}}

\newcommand{\beq}{\begin{equation}}
\newcommand{\eeq}{\end{equation}}
\newcommand{\beqa}{\begin{eqnarray}}
\newcommand{\eeqa}{\end{eqnarray}}
\newcommand{\bit}{\begin{itemize}}
\newcommand{\eit}{\end{itemize}}
\newcommand{\bedef}{\begin{defn}}
\newcommand{\edefn}{\end{defn}}
\newcommand{\bpro}{\begin{prop}}
\newcommand{\epro}{\end{prop}}

\newcommand{\df}{\partial}
\newcommand{\bv}{{\bf v}}
\newcommand{\bx}{{\bf x}}
\newcommand{\by}{{\bf y}}
\newcommand{\bq}{{\bf q}}

\newcommand{\bu}{{\bf u}}
\newcommand{\bn}{{\bf n}}
\newcommand{\mE}{{\mathcal E}}
\newcommand{\mT}{{\mathcal T}}
\newcommand{\mG}{{\mathcal G}}
\newcommand{\mU}{{\mathcal U}}

\newcommand{\veps}{{\varepsilon}}

\def\signFF{\bigskip\bigskip\hspace{80mm}
\vbox{{\sc Francis Filbet\par\vspace{3mm}
Universit\'e de Toulouse III \& IUF \par
Institut de Math\'ematiques de Toulouse,\par
118, route de Narbonne\par
F-31062 Toulouse cedex,  FRANCE
\par\vspace{3mm}e-mail:} francis.filbet@math.univ-toulouse.fr }}

\def\signCWS{\bigskip\bigskip\hspace{80mm}
\vbox{{\sc Chi-Wang Shu \par\vspace{3mm}
Division of Applied Mathematics \par 
Brown University \par 
Providence, RI 02912, USA
\par\vspace{3mm}e-mail:}  shu@dam.brown.edu}}

\begin{document}

\title[Numerical methods  for a kinetic model of
  self-organized dynamics]{Discontinuous-Galerkin methods for a kinetic model of
  self-organized dynamics}

\begin{abstract}
This paper deals with the numerical resolution of kinetic models for
systems of self-propelled particles subject to alignment interaction
and attraction-repulsion. We focus on the kinetic model considered in
\cite{DM, DLMP} where alignment is taken into account in addition of
an attraction-repulsion interaction potential.  We apply a
discontinuous Galerkin method for the free transport and non-local
drift velocity together with a spectral method for the velocity
variable. Then, we analyse
consistency and stability of the semi-discrete scheme. We propose several numerical experiments which provide a solid validation of the method and its underlying concepts.
\end{abstract}

\author{Francis Filbet and Chi-Wang Shu}

\maketitle


\vskip 0.4cm

\tableofcontents

\section{Introduction} 
\setcounter{equation}{0}
\label{sec:1}

Theoretical and mathematical biology communities have paid a great
deal of attention to explain large scale structures in animal
groups. Coherent structures appearing from seemingly direct
interactions between individuals have been reported in many different
species like fishes, birds, and insects \cite{HW,parrish,
  mogilner,BDT,Couzin,camazine} and many others, see also the reviews
\cite{review,review2,kolo}. There has been an intense literature about the modeling of
interactions between individuals among animal societies such as fish
schools, bird flocks, herds of mammalians, etc. We refer for instance to
\cite{Aldana_Huepe, Aoki, Couzin, chate} but an exhaustive
bibliography is out of reach. Among these models, the Vicsek model
\cite{Vicsek} has received particular attention due to its simplicity
and the universality of its qualitative features. This model is an
individual based model  or agent-based
model which consists of a time-discretized set of Ordinary
Differential Equations for the particle positions and velocities. 
A time-continuous version of this model and its kinetic formulation are
available in \cite{DM}.  A rigorous derivation of this kinetic model
from the time-continuous Vicsek model can be found in \cite{BCC} and
in \cite{DLMP} when adding an attraction-repulsion force.

On the other hand, hydrodynamic models are attractive over particle
ones due to their computational efficiency. For this reason, many such
models have been proposed in the literature \cite{CDP, CDMBC, DCBC,
  mogilner, TB1, TB2}. However, most of them are phenomenological. For
instance in \cite{DM}, the authors  propose one of the first rigorous
derivations of a hydrodynamic version of the Vicsek model (see also
\cite{KRZB,RBKZ,RKZB} for phenomenological derivations). It has been
expanded in \cite{DM2} to account for a model of fish behavior where
particles interact through curvature control, and in \cite{DY} to
include diffusive corrections. Other variants have also been
investigated \cite{DFL,Frouvelle,FL}. For instance, \cite{Frouvelle}
studies the influence of a vision angle and of the dependency of the
alignment frequency upon the local density, whereas in \cite{DFL,FL},
the authors study a modification of the model which results in phase transitions from disordered to ordered equilibria as the density increases and reaches a threshold, in a way similar to polymer models \cite{DE, Onsager}.

In this paper, we will focus on the numerical approximation of a kinetic model for self-propelled
particles. The self-propulsion speed is supposed to be constant and
identical for all the particles. Therefore, the velocity variable
reduces to its orientation in the ($d-1$)-dimensional sphere ${\mathbb S}^{d-1}$.
The particle interactions consist in two parts: 
\begin{itemize}
\item an alignment rule which tends to relax the particle velocity to
  the local average orientation; 
\item an attraction-repulsion rule which makes the particles move
  closer or farther away from each other. 
\end{itemize}

This model is inspired both by the Vicsek model \cite{Vicsek} and the
Couzin model \cite{Aoki, Couzin} describing interactions at the
microscopic level. This approach has led to different types of models for swarming: microscopic models and macroscopic models involving macroscopic quantities (e.g. mass,
flux). Here we study an intermediate approach, called the
mesoscopic scale, where we investigate the time evolution of a
distribution function of particles $f(t,\bx,\bv)$, depending on time
$t\geq 0$, position $\bx\in\Omega\subset\RR^d$ and velocity
$\bv\in{\mathbb S}^{d-1}$. This
distribution function is solution to a kinetic equation describing
the motion of particles and their change of directions. Several works
have already studied kinetic models for swarming \cite{CCR,HT08,CDP},
but few have done a numerical investigation. For macroscopic and
microscopic models, we refer to \cite{DM, Dimarco, Dimarco2} for
discrete particle approximation of microscopic models and finite volume or
finite difference approximations \cite{Dimarco2} where complex
structures may be observed as in \cite{FFCWS, Dimarco2}.  Concerning
kinetic model, we refer to  \cite{gamba2} for the first work on this
topic, where the
authors propose a spectral discretization of the operator describing
the change of direction with the flavor of what has been done for the
Boltzmann equation \cite{rey, FHJ}. This spectral discretization is coupled with a finite volume approximation for the
transport \cite{F1, FR03,FM11,FY}. For the study of bacteria's motion,
we also mention \cite{FY13} where similar structures have been
observed (band's formation). Here we present a local discontinuous Galerkin method for
computing the approximate solution to the kinetic model and to get high order approximations in time,
space and velocity. Indeed, the preservation of high order accuracy allows to
investigate complex structures in space as it has already been
observed for macroscopic models \cite{Dimarco,Dimarco2}.   

The paper is organized as follows: we first present precisely the
kinetic model and give the main assumptions on the regularity of
the unique solution to prove convergence and error estimates on the
approximation to the exact smooth solution. Then, in Section \ref{sec:2} we develop a numerical scheme
(local discontinuous Galerkin method) for the kinetic model. In
Sections \ref{sec:3} and \ref{sec:4}, we perform a stability and
convergence analysis of the proposed numerical methods. Numerical
investigations are presented in Section \ref{sec:5} where the order accuracy is
verified and we observe the formation of complex structures.

\subsection{Agent-based model of self-alignment with attraction-repulsion}
\label{sec:1.1}

The starting point of this study is an Individual-Based Model of
particles interacting through self-alignment \cite{Vicsek} and
attraction-repulsion \cite{Aoki, Couzin}. Specifically, we consider
$N$ particles $\bx_i\in\mathbb{R}^d$, with $d=2$ or $3$,  moving at a
constant speed $\bv_i \in\mathbb{S}^{d-1}$. Each particle adjusts its
velocity to align with its neighbors and to get closer or further
away. Therefore, the evolution of each particle is modeled by the
following dynamics: for any $1\leq i \leq N$ 
\begin{equation}
  \label{eq:particle_1}
\left\{\begin{array}{l}
 \ds \frac{d\bx_i}{dt} = \bv_i,
 \\
\,
\\
\ds d\bv_i = \bP_{\bv_i^\bot} \big(\overline{\bv}_i\,dt \,+\,
\sqrt{2\nu}\;\, d \bB_t^i \big),
\end{array}\right.
\end{equation}
where $\bB_t^i$ is a Brownian motion and $d$ represents the noise
intensity whereas $\bP_{\bv_i^\bot}$ is the projection matrix onto the normal plane to $\bv_i$:

\begin{displaymath}
  \bP_{\bv^\bot} = \mbox{Id} - \bv \otimes \bv,
\end{displaymath}
which ensures that $\bv_i$ stays of norm $1$. 

Both the alignment and attraction-repulsion rules are taken into
account in the macroscopic velocity $\overline{\bv}_i\in \mathbb{S}^{d-1}$:
\begin{displaymath}
  \overline{\bv}_i \,=\, \frac{1}{|\bJ_i + \bR_i|}\,(\bJ_i \,+\, \bR_i),
\end{displaymath}
where $\bJ_i$ counts for the alignment and $\bR_i$ for the attraction-repulsion:
\begin{equation}
  \label{eq:j_r_particle}
  \bJ_i \,=\, \sum_{j=1}^N k(|\bx_j-\bx_i|) \,\bv_j, \qquad \bR_i \,=\, -\sum_{j=1}^N 
  \nabla_{\bx_i} \phi(|\bx_j-\bx_i|),
\end{equation}
where the kernel $k$ is a positive function, $\phi'$ can be both negative
(repulsion) and positive (attraction) and for simplicity, we will
assume that both $k$ and $\phi$ are compactly supported in $[0,\infty)$.

\subsection{Kinetic model of self-alignment with attraction-repulsion}
\label{sec:1.2}

When the number of particles becomes large, that is
$N\rightarrow\infty$, one can formally derive a Vlasov type
equation. It describes the evolution of a system of particles under the effects of external and
self-consistent fields. The unknown $f(t,\xx,\vv)$, depending on the
time $t$, the position $\xx$, and the velocity $\vv$, represents the
distribution of particles in phase space for each species with
$(\xx,\vv) \in \Omega\times \SSS^{d-1}$, $d=1,..,3$, where $\Omega\subset\RR^d$. Its behaviour is
given by the Vlasov equation \cite{BCC,DM,sznitman_topics_1989,DLMP}, 
\begin{equation}
\frac{\partial f}{\partial t} \,+\, \bv \cdot \nabla_\bx f = - {\rm div}_\bv\left[ \bP_{\bv^\bot} \bv_f\, f - \nu \,\nabla_\bv f\right],
\label{kinetic:eq}
\end{equation}
where $\nu>0$ and
\begin{equation}
\left\{
\begin{array}{l}
\ds\bv_f = \frac{1}{|\bJ_f+\bR_f|}\,(\bJ_f+\bR_f),  
\\
\,
\\
 \ds\bJ_f = \int_{\Omega\times\SSS^{d-1}} k(|\bx - \bx'|) \bv' \,
 f(t,\bx',\bv') \, d\bx'd\bv', 
\\
\;
\\ 
\ds\bR_f = - \nabla_\bx \int_{\Omega\times\SSS^{d-1}}\phi(|\bx - \bx'|) \, f(t,\bx',\bv') \, d\bx'd\bv'.
\end{array}\right.
\label{vel:eq}
\end{equation}
In general, the function $\phi$ is such that $\phi(r) \rightarrow 0$ when $r
\rightarrow \infty$, but here we will assume that both $k$ and $\phi$
are nonnegative functions which  satisfy
\beq
\label{hyp:00}
k, \phi \in {\mathcal C}^p_c([0,\infty)), \quad{\rm with }\, p\geq 2
\eeq
and  for periodic boundary conditions in space, we have
$$
\left\{\begin{array}{l}
\ds\bJ_f(t,\bx) \,=\, \int_{{\rm supp}(k)} k(|\by|) \,\rho\,\bu(t,\bx+\by) \, d\by, 
\\ \,\\
\ds \bR_f(t,\bx) \,=\, -\nabla_\bx\int_{{\rm supp}(\phi)} \phi(|\by|)\, \rho(t,\bx+\by) \, d\by. 
\end{array}\right.
$$
Furthermore we assume that the system
(\ref{kinetic:eq})-(\ref{vel:eq}) has a smooth solution such that
$$
f \in H^{k+2}([0,T]\times\Omega\times\SSS^{d-1}),
$$
with $\bJ_f$ and $\bR_f$ such that for any $T >0$, there exists a
constant $\xi_T>0$ such that for all $(t,\bx)\in [0,T]\times \Omega$
\beq
\label{hyp:01}
|\bJ_f(t,\bx) + \bR_f(t,\bx) | \geq \xi_T.
\eeq

Using the distribution $f$ and a rescaling of
(\ref{kinetic:eq}), it is possible to identify the
asymptotic behavior of the model in different regimes as in
\cite{DLMP} and to recover some classical hydrodynamic model for the
self-organized dynamics. 

The model studied in this paper is a generalization of the model of
\cite{DM} with the addition of an attraction-repulsion interaction
potential \cite{DLMP}.

\section{Numerical Methods}
\setcounter{equation}{0}
\label{sec:2}

In this section, we will introduce the discontinuous Galerkin
algorithm for the system (\ref{kinetic:eq})-(\ref{vel:eq}).   
Discontinuous Galerkin methods are particularly suited
for transport type equations with several attractive
properties, such as their easiness for adaptivity and
parallel computation, and their nice stability properties.
We refer to the survey paper \cite{CS} and the references
therein for an introduction to discontinuous Galerkin methods.
For discontinuous Galerkin methods solving kinetic type equations
we refer to \cite{CGMS,ACS}.
We consider an open set $\Omega\subset \RR^d$ where all boundary
conditions are  periodic, and $f(t,\bx,\bv)$ is assumed to be in the
unit sphere $\SSS^{d-1}$. 

\subsection{Notations}
\label{sec:2.1}

Let $\mT_h^\bx=\{K_\bx\}$ and $\mT_h^\bv=\{K_\bv\}$ be  partitions of
$\Omega$ and $\SSS^{d-1}$, respectively,  with $K_\bx$ and $K_\bv$
being  Cartesian elements
;  then 
$$
\mT_h=\left\{K\subset \Omega\times\SSS^{d-1}; \quad K=K_\bx\times K_\bv\in\mT_h^\bx\times \mT_h^\bv\right\}
$$ 
defines a partition of $\Omega\times\SSS^{d-1}$.

Let $\mE$ be the set of the edges of $\mT_h$ will
be $\mE=\mE_\bx \cup \mE_\bv$ as
$$
\left\{\begin{array}{l}
\mE_\bx\,=\, \left\{\sigma_\bx\times K_\bv: \quad
\sigma_\bx\in\partial K_\bx, \, K_\bv\in\mT_h^\bv\right\},
\\
\,
\\
\mE_\bv\,=\,\left\{ K_\bx\times \sigma_\bv: \quad K_\bx\in\mT_h^\bx,
\,\sigma_\bv\in \partial K_\bv\right\}.
\end{array}\right.
$$

Next we define the discrete spaces
\begin{equation}
\label{G:space}
\mG_h^{k} \,\,=\,\,\left\{g\in L^2(\Omega\times\SSS^{d-1}):\quad  g|_K\in
  P^k(K), \, K\in\mT_h \right\},
\end{equation}
and 
$$
\mU_h^{k}\,\,=\,\,\left\{\bU\in [L^2(\Omega\times\SSS^{d-1})]^{d-1}: \quad \bU|_K\in [P^k(K)]^{d-1}, \, K\in\mT_h \right\},
$$
where $P^{k}(K)$ denotes the set of polynomials of  total degree at
most $k$ on $K$, and $k$  is  a nonnegative integer.

Note the space $\mG_h^{k}$, which  we use to approximate $f$,  is called P-type, and it can be replaced by the tensor product of P-type spaces in $\bx$ and $\bv$,
$$
\left\{g\in L^2(\Omega\times\SSS^{d-1})\,: \,\,g|_{K}\in
  P^k(K_\bx)\times P^k(K_\bv), \, K=K_\bx\times K_\bv\in\mT_h \right\},
$$
or by the tensor product space in each variable, which is called  Q-type 
$$
\left\{g\in L^2(\Omega\times\SSS^{d-1})\,:\,\, g|_{K}\in
  Q^k(K_\bx)\times Q^k(K_\bv), \, K=K_\bx\times K_\bv\in\mT_h \right\}.
$$
Here $Q^k(K)$ denotes the set of polynomials of    degree at most $k$
in each variable on $K$. The numerical methods formulated in this
paper, as well as the conservation, stability, and error estimates,
hold when any of the spaces above is used to approximate $f$. 

\begin{remark}
In our simulations of Section \ref{sec:5}, we use the P-type of
\eqref{G:space} as it is the smallest and therefore renders the most
cost efficient algorithm. 
\end{remark}
 
For piecewise  functions defined with respect to $\mT_h^\bx$ or
$\mT_h^\bv$, we further introduce the jumps and averages as
follows. For $\alpha\in\{\bx,\bv\}$ and for any edge $\sigma=\{K_\alpha^+\cap K_\alpha^-\}\in\mE_\alpha$, with $\bn_\alpha^\pm$ as the outward unit normal to $\partial K_\alpha^\pm$,
$g^\pm=g|_{K_\alpha^\pm}$,  the jumps across $\sigma$ and the averages are defined  as
\begin{equation}
\label{gplus}
[g]_\alpha={g^+}-{g^-},\qquad
\{g\}_\alpha=\frac{1}{2}({g^+}+{g^-}),\quad \alpha\in\{\bx,\bv\}.
\end{equation}

\subsection{The semi-discrete discontinuous Galerkin  method}
\label{sec:2.1.}

The numerical methods proposed in this section are formulated for the
system (\ref{kinetic:eq})-(\ref{vel:eq}).  Given $k, r\geq 0$, the
semi-discrete discontinuous Galerkin methods for  the
system (\ref{kinetic:eq})-(\ref{vel:eq}) are defined by the  following
procedure: for any $K=K_\bx\times K_\bv\in\mT_h$, we look for
$(f_h,\bq_h)\in\mG_h^k\times \mU_h^k$, $\bv_{f_h}\in\mU_h^r$, such that
for all $g\,\in\,\mG_h^k$, 
\begin{eqnarray}
&&\int_K\frac{\partial f_h}{\partial t}  g \,d\bx d\bv - \int_K f_h\bv\cdot\nabla_\bx g \,d\bx d\bv
- \int_K ( \bP_{\bv^\perp}  \bv_{f_h}\,f_h\,-\, \nu\,\bq_h)\cdot\nabla_\bv g \,d\bx
d\bv
\label{eq:scheme:1a}
\\
\nonumber
&+& \,\int_{\sigma_\bx} \widehat{f_h \,\bv}\cdot \bn_x\,g^-
\,ds_\bx \,d\bv \,+\, \int_{\sigma_\bv} \left(
  \widehat{f_h \,\bP_{\bv^\perp}\bv_{f_h}} \,-\,\nu\,\widehat{\bq_h}\right)\cdot \bn_\bv \, g^- \,ds_\bv d\bx
\,=\,0, \,
\end{eqnarray}
where $\bn_\bx$ and $\bn_\bv$ are outward unit normals of $\df K_\bx$
and $\df K_\bv$, respectively, whereas $\bq_h$ is given by
\begin{equation}
\label{eq:scheme:1b}
\ds\int_K \bq_h \cdot\bu \,d\bx \,d\bv \,+\, \int_K f_h\,{\rm div}_{\bv}\bu \,d\bx
d\bv -  \int_{\sigma_\bv}
\widehat{ f_h} \,\bn_\bv \cdot \bu^- \,d\bx\,ds_\bv \,=\, 0,\,\forall\bu\,\in\, \mU_h^k. 
\end{equation}
Furthermore,  the velocity $\bv_{f_h} \in
L^\infty(\Omega)$ with $\|\bv_{f_h}\|=1$, and  
\beq
\label{eq:scheme:2}
{\bv}_{f_h}(t,\bx) \,=\, \frac{1}{\|\bJ_{h}(t,\bx)+\bR_{h}(t,\bx)\|}\,(\,\bJ_{h}(t,\bx)+\bR_{h}(t,\bx)\,),
\eeq
with $\bJ_h$ and $\bR_h$  computed by
\beq
\left\{
\begin{array}{l}
\ds \bJ_{h}(t,\bx)  = \int_{\Omega} k(|\bx-\bx'|)\, \rho_h \bu_h(t,\bx')
 \,d\bx',
\\
\,
\\
\ds\bR_{h} = \int_{\Omega} \nabla_{\bx}\phi (|\bx-\bx'|) \,\rho_h(t,\bx') d\bx',
\end{array}\right.
\label{eq:scheme:3}
\eeq
where $\rho_h$ and $\bu_h$ are defined from the distribution function
$f_h$, by
$$
\rho_h \,=\, \int_{\SSS^{d-1}} f_h d\vv, \quad \rho_h\bu_h \,=\,
\int_{\SSS^{d-1}} \bv \,f_h d\vv.
$$
 All  hat  functions are numerical
fluxes that are determined by upwinding for convection
and local DG alternating for diffusion, {\it i.e.} for the
convective terms in (\ref{eq:scheme:1a})
\begin{equation}
\label{eq:flux:1a}
\left\{\begin{array}{l}
\ds \widehat{f_h \bv} \cdot
  \bn_\bx \,\,=\,\,  \bv \cdot \bn_\bx \,\{f_h \}_\bx \,-\,\frac{|\bv\cdot\bn_\bx|}{2}[f_h]_\bx,
\\
\,
\\
\ds\widehat{f_h\,\bP_{{\bv}^\perp} \bv_{f_h}} \cdot
  \bn_\bv\,\,=\,\,  \bP_{{\bv}^\perp} \bv_{f_h}\cdot
  \bn_\bv   \, \{f_h\}_\bv \,-\,\frac{|\bP_{{\bv}^\perp} \bv_{f_h}|}{2}[f_h]_\bv,
\end{array}\right.
\end{equation}
and for the diffusive terms in (\ref{eq:scheme:1b}), we apply 
\begin{equation}
\label{eq:flux:2}
\ds\widehat{\bq_h}\cdot \bn_\bv
\,\,=\,\, \{\bq_h\}\cdot\bn_\bv \,+\, \frac{C_{11}}{2} [f_h]_\bv,
\qquad
\ds\widehat{f_h}  \bn_\bv\,\,=\,\, \{f_h\} \,\bn_\bv \,+\, 
\frac{C_{22}}{2}\, [\bq_h]_\bv
\end{equation}
where $C_{11}, C_{22} >0$. 

This completes the definition of our Discontinuous-Galerkin method,
but to facilitate its study, we recast its formulation. We sum (\ref{eq:scheme:1a}) and (\ref{eq:scheme:1b}) over all elements and define $a_h(.)$ and $b_h(.)$ such that for
$(f_h,\bq_h)\in\mG_h^k\times \mU_h^k$ and  $g\,\in\,\mG_h^k$
\begin{eqnarray}
\label{def:ah}
a_h(f_h,\bq_h,\bv_{f_h}, g) &:=&
                                 \int_{\Omega\times\SSS^{d-1}}\left(\frac{\partial f_h}{\partial t}  g -  f_h\bv\cdot\nabla_\bx g \right)\,d\bx d\bv
\\
&-& \int_{\Omega\times\SSS^{d-1}}\left(  \bP_{\bv^\perp}  \bv_{f_h}\,f_h\,-\, \nu\,\bq_h\right)\cdot\nabla_\bv g\,d\bx d\bv
\nonumber
\\
&-& \,\sum_{\sigma_\bx\in \mE_\bx}\int_{\sigma_\bx} \widehat{f_h \,\bv}\cdot \bn_x\,[g]_\bx
\,ds_\bx \,d\bv 
\nonumber
\\
&-& \,\sum_{\sigma_\bv\in \mE_\bv}\int_{\sigma_\bv} \left(
  \widehat{f_h \,\bP_{\bv^\perp}\bv_{f_h}}
    \,-\,\nu\,\widehat{\bq_h}\right)\cdot \bn_\bv \, [g]_\bv \,ds_\bv
    d\bx,
\nonumber
\end{eqnarray}
and for $\bu\,\in\, \mU_h^k$ 
\begin{eqnarray}
\label{def:bh}
b_h(f_h,\bq_h, u) &:=& \int_{\Omega\times\SSS^{d-1}} \left(\bq_h
                       \cdot\bu \,+\, f_h\,{\rm div}_{\bv}\bu \right)\,d\bx
d\bv 
\\
\nonumber
&+& \sum_{\sigma_\bv\in \mE_\bv}\int_{\sigma_\bv}\widehat{ f_h} \,\bn_\bv \cdot [\bu]_\bv \,d\bx\,ds_\bv,
\end{eqnarray}
where we  notice that  $a_h$ (resp. $b_h$) is linear with respect to
$(f_h,\bq_h)$ and $g$ (resp. $\bu$). 

We prove the following boundedness and error estimate results.

\begin{theorem}
Assume that the solution to (\ref{kinetic:eq})-(\ref{vel:eq}) is such that $f\in
H^{k+2}([0,T]\times\Omega\times\SSS^{d-1})$ with the two hypothesis
(\ref{hyp:00}) and (\ref{hyp:01}). We also assume  that the initial data $f_h(0)$ is uniformly bounded in
$L^2(\Omega\times\SSS^{d-1})$ and for $k\geq 0$, we consider the numerical solution $(f_h,\bq_h)\in\mG_h^k\times
\mU_h^k$ given by (\ref{eq:scheme:1a})-(\ref{eq:flux:2}) supplemented
with periodic boundary conditions. Then, for any $h_0>0$, there exists $C_T>0$,
depending on $f$,  $T$ and $h_0$, such that  for $h<h_0$
$$
\|f_h(t)\|_{L^2}^2 \,+\, \int_0^t\|\bq_h(s)\|_{L^2}^2\,ds \,\leq\, C_T, \quad t\in [0,T]
$$
and 
$$
\|f(t)- f_h(t)\|_{L^2} \,+\,  \left(\int_0^t\|\bq(s)-\bq_h(s)\|_{L^2}^2 ds\right)^{1/2} \,\leq\, C_T\, h^{k+1/2}, \quad t\in [0,T].
$$
\label{th:1}
\end{theorem}

\subsection{Temporal discretizations}
\label{sec:2.2}
We use total variation diminishing (TVD) high-order Runge-Kutta
methods to solve the method of lines  ordinary differential equation resulting from the
semi-discrete discontinuous Galerkin scheme, 
$$
\frac{d f_h}{dt} \,\,=\,\,{\mathcal R}(f_h).
$$
Such time stepping methods are convex combinations of the Euler
forward time discretization. The commonly used third-order TVD
Runge-Kutta method  is given by
\begin{equation}
\label{3rk:1}
\left\{
\begin{array}{l}
\ds f_h^{(1)}\,=\,f_h^n\,\,+\,\,\triangle t \,\mathcal{R} (f_h^n), 
\\
\, 
\\
\ds f_h^{(2)}\,=\;\frac{1}{4} \,\left( 3\, f_h^n\,\,+\,\,f_h^{(1)}\,\,+\,\, \triangle t \,\mathcal{R}(f_h^{(1)}) \right),
\end{array}\right.
\end{equation}
and 
\begin{equation}
\label{3rk:2}
f_h^{n+1} \,=\,\frac{1}{3} \left( f_h^n\,+\, 2\,f_h^{(2)} \,+\,
\,{2}\,\triangle t \,\mathcal{R}(f_h^{(2)})\right),
\end{equation}
where $f_h^n$ represents a numerical approximation of the solution at
discrete time $t_n$.  

A detailed description of the TVD Runge-Kutta method can be found in
\cite{Shu_1988_JCP_NonOscill}; see also \cite{Gottlieb_1998_MC_RK} and
\cite{Gottlieb_2001_SIAM_Stabi}  for strong-stability-preserving methods.

\section{Conservation and stability}
\setcounter{equation}{0}
\label{sec:3}

In this section, we will establish  conservation and stability
properties of  the semi-discrete discontinuous Galerkin methods.
In particular, we prove that for periodic  boundary condition, the total density  (mass) is always conserved.  We also show that $f_h$ is $L^2$ stable, which facilitates the error analysis of Section \ref{sec:4}.    

\begin{lemma}[{Mass conservation}]
Consider the numerical solution $(f_h,\bq_h)\in\mG_h^k\times \mU_h^k$ for
$k\geq 0$ given by (\ref{eq:scheme:1a})-(\ref{eq:flux:2}) supplemented
with periodic boundary conditions. Then it satisfies
\begin{equation}
\frac{d}{dt}\int_{\Omega\times\SSS^{d-1}} f_h d\bx d\bv\,=\,0.
\label{eq:MassC}
\end{equation}
Equivalently, for $\rho_h(\bx, t)$, for any $t>0$, the following holds: 
\beq
\int_{\Omega} \rho_h(t,\bx)d\bx = \int_{\Omega} \rho_h(0,\bx) d\bx.
\label{eq:MassC:1}
\eeq
\label{lem:1}
\end{lemma}

 \begin{proof}
 Let $g(\bx, \bv)=1$ and noticing that  $g\in\mathcal{G}_h^k$, for any $k\geq 0$,  is continuous  
$\nabla_\bx g=0$ and $\nabla_\bv g=0$. Taking  this $g$ as the test function in \eqref{eq:scheme:1a}, one has
$$
\frac{d}{dt}\int_{K} f_h d\bx d\bv  \,+\,  
 \int_{\sigma_\bx} \widehat{f_h \,\bv}\cdot \bn_x 
\,ds_\bx \,d\bv \,+\, \int_{\sigma_\bv} \left(\widehat{f_h \,\bP_{\bv^\perp}\bv_{f_h}}  \,-\,\nu\,\widehat{\bq_h}  \right) \cdot \bn_\bv\,ds_\bv d\bx \,=\, 0.
$$
Then summing over $K\in \mT_h$ and thanks to the periodic boundary
conditions, we get the conservation of total
mass for any $t\geq 0$,
$$
\frac{d}{dt}\int_{\Omega\times\SSS^{d-1}} f_h(t) d\bx d\bv \,=\, 0.
$$
Finally from the definition of $\rho_h$ and  integrating in time from
$0$ to $t$, it gives \eqref{eq:MassC:1}.
\end{proof}

Finally, we can obtain the $L^2$-stability result for $f_h$.  This result will  be used in the error analysis of Section \ref{sec:4}. 
\begin{lemma}[{$L^2$-stability of $f_h$}]
Assume that the initial data $f_h(0)$ is uniformly bounded in
$L^2(\Omega\times\SSS^{d-1})$ and for $k\geq 0$, consider the numerical solution $(f_h,\bq_h)\in\mG_h^k\times
\mU_h^k$ given by (\ref{eq:scheme:1a})-(\ref{eq:flux:2}) supplemented
with periodic boundary conditions. Then  $(f_h,\bq_h)$ satisfies for
any $t\geq 0 $
\begin{eqnarray*}
\frac{1}{2}\frac{d}{dt}\int_{\Omega\times\SSS^{d-1}} |f_h|^2 \,d\bx
   d\bv &+& \nu \int_{\Omega\times\SSS^{d-1}} |\bq_h|^2 \,d\bx d\bv  
\\
  &+& \frac{1}{2}\,\sum_{\sigma_\bv\in\mE_\bv}\int_{\sigma_\bv} \left(\left| \bP_{\bv^\perp}\bv_{f_h}\cdot
    \bn_\bv\right|+\nu\,C_{11}\right)\, [f_h]_\bv^2\,+\, \nu\,C_{22} \, [\bq_h]_\bv^2 \,ds_\bv\,d\bx
\\
&+& \frac{1}{2}\,\sum_{\sigma_\bx\in\mE_\bx}\int_{\sigma_\bx}
    |\bv\cdot \bn_x|\,[f_h]_\bx^2\,ds_\bx\,d\bv  \\
&\,=\,&  -\frac{1}{2}\,\int_{\Omega\times\SSS^{d-1}} f_h^2
                                                      \, {\rm div}_\bv\left( \bP_{\bv^\perp}\bv_{f_h}\right) \,d\bx d\bv.
\end{eqnarray*}
\label{lem:stability}
\end{lemma}
\begin{proof}
Observing  that
\begin{eqnarray*}
 \int_{K_\bv} f_h\, \bP_{\bv^\perp}\bv_{f_h}  \cdot\nabla_\bv
  f_h \, d\bv &=&\frac{1}{2}\,\int_{K_\bv} \bP_{\bv^\perp}\bv_{f_h}  \cdot\nabla_\bv
  f_h^2 \, d\bv,
\\
 &=&  \frac{1}{2}\,\int_{K_\bv}{\rm div}_\bv\left( \bP_{\bv^\perp}\bv_{f_h}
  f_h^2\right)\, -\,  f_h^2\,{\rm div}_\bv\left( \bP_{\bv^\perp}\bv_{f_h}\right)\, d\bv,
\\
&=& \frac{1}{2}\,\int_{\df K_\bv} \bP_{\bv^\perp}\bv_{f_h}\cdot \bn_\bv
\,|f_h^-|^2 \,ds_\bv  \, -\,  \frac{1}{2}\,\int_{K_\bv} f_h^2 \, {\rm div}_\bv\left( \bP_{\bv^\perp}\bv_{f_h}\right)\, d\bv, 
\end{eqnarray*}
and summing over all control volume $K\in\mT^h$ and recasting the
edges, it yields
\begin{eqnarray*}
\int_{\Omega\times\SSS^{d-1}} f_h\, \bP_{\bv^\perp}\bv_{f_h}  \cdot\nabla_\bv
  f_h \, d\bv d\bx  &=&  - \frac{1}{2}\,\int_{\Omega\times\SSS^{d-1}} f_h^2 \, {\rm
    div}_\bv\left( \bP_{\bv^\perp}\bv_{f_h}\right)\, d\bv d\bx \\
&-& \sum_{\sigma_\bv\in\mE_\bv}\int_{\sigma_\bv} \bP_{\bv^\perp}\bv_{f_h}\cdot \bn_\bv\,\{f_h\}_\bv\,[f_h]_\bv \,ds_\bv d\bx.
\end{eqnarray*}
Moreover, we have after recasting
$$
\int_{\Omega\times\SSS^{d-1}} f_h\, \bv  \cdot\nabla_\bx
  f_h \, d\bv d\bx \, =\, - \sum_{\sigma_\bx\in\mE_\bx}\int_{\sigma_\bx} \bv\cdot \bn_\bx\,\{f_h\}_\bx\,[f_h]_\bx \,ds_\bv d\bx.
$$

Then we take  $g=f_h$  in \eqref{def:ah}, it
gives after an integration by part on each control volume
\beq
\label{i:0}
\frac{1}{2}\frac{d}{dt}\int_{\Omega\times\SSS^{d-1}} |f_h|^2 \,d\bx d\bv + \frac{1}{2}\,\int_{\Omega\times\SSS^{d-1}} f_h^2
                                                      \, {\rm div}_\bv\left( \bP_{\bv^\perp}\bv_{f_h}\right)
\,  d\bx d\bv \,+\,  I_1 + I_2 + I_3  = 0,
\eeq
with 
$$
\left\{
\begin{array}{l}                                                       
\ds I_1 =  -\sum_{\sigma_\bx\in \mE_\bx}\int_{\sigma_\bx} \left(
  \widehat{f_h \bv}\cdot \bn_x \,-\, \bv \cdot \bn_x \{ f_h\}_\bx \right) \,[f_h]_\bx \,ds_\bx \,d\bv, 
\\ \, \\
\ds I_2 = -\sum_{\sigma_\bv\in\mE_\bv} \int_{\sigma_\bv} 
  \left( \widehat{f_h \,\bP_{\bv^\perp}\bv_{f_h}}\cdot \bn_\bv  \,-\,
  \bP_{\bv^\perp}\bv_{f_h}\cdot \bn_\bv\,\{f_h\}_\bv\right) \,[f_h]_\bv
\,ds_\bv d\bx,
\\ \,\\
\ds I_3 = \nu \left( \int_{\Omega\times\SSS^{d-1}}\bq_h\cdot
  \nabla_\bv f_h d\bv d\bx + \sum_{\sigma_\bv\in\mE_\bv} \int_{\sigma_\bv} \widehat{\bq_h}\cdot \bn_\bv\,[f_h]_\bv \,ds_\bv d\bx\right).
\end{array}\right.
$$
Let us prove that each term $I_k$, for $1\leq k \leq 3$, is
nonnegative.  On the one hand  using the the definition of the upwinding flux (\ref{eq:flux:1a}), we simply have  
\beq
\label{i:1}
I_1 \,=\,  \frac{1}{2}\,\sum_{\sigma_\bx\in\mE_\bx}\int_{\sigma_\bx}
|\bv\cdot \bn_x|\,[f_h]_\bx^2\,ds_\bx\,d\bv \geq 0
\eeq
and
\beq
\label{i:2}
I_2 \,=\,  \frac{1}{2} \,\sum_{\sigma_\bv\in\mE_\bv}\int_{\sigma_\bx}
\left| \bP_{\bv^\perp}\bv_{f_h}\cdot \bn_\bv\right|\,
[f_h]_\bv^2\,ds_\bv\,d\bx\geq 0.
\eeq
On the other hand, to deal with the last term $I_3$, we  choose
$\bu=\bq_h$  in \eqref{def:bh}, hence we get 
$$
\nu \left( \int_{\Omega\times\SSS^{d-1}} \left(\,|\bq_h|^2 \,+\,f_h\,{\rm div}_{\bv} \bq_h\right) \,d\bx
d\bv \,+\,  \sum_{\sigma_\bv\in\mE_\bv} \int_{\sigma_\bv}
\widehat{ f_h} \,\bn_\bv \cdot [\bq_h]_\bv \,d\bx\,ds_\bv\right) \,=\, 0.
$$
Then performing an integration by part in
velocity of the second term and using the definition of $I_3$, we have 
\begin{eqnarray*}
I_3 &=& \nu   \int_{\Omega\times\SSS^{d-1}} |\bq_h|^2 \,d\bx \,d\bv \\
&+&  \nu\sum_{\sigma_\bv\in\mE_\bv} \int_{\sigma_\bv}
 (f_h^-\,\bq_h^- \,-\,f_h^+\bq_h^+)\cdot\bn_\bv + \widehat{\bq_h}\cdot \bn_\bv\,[f_h]_\bv+\widehat{ f_h} \,\bn_\bv \cdot [\bq_h]_\bv \,d\bx\,ds_\bv.
\end{eqnarray*}
Therefore from the definition of the  ``alternating fluxes''
(\ref{eq:flux:2}), we finally get that 
\beq
\label{i:3}
I_3 = \nu \left( \int_{\Omega\times\SSS^{d-1}}  |\bq_h|^2 \,d\bx
  \,d\bv  \,+\, \frac{1}{2}\,\sum_{\sigma_\bv\in\mE_\bv} \int_{\sigma_\bv} C_{11}
  \,[f_h]_\bv^2 \,+\, C_{22}\, [\bq_h]_\bv^2   \right)\geq 0. 
\eeq
Gathering  (\ref{i:0}) together with  (\ref{i:1})-(\ref{i:3}), we
obtain the result
\begin{eqnarray*}
&&\frac{1}{2}\,\frac{d}{dt}\int_{\Omega\times\SSS^{d-1}} |f_h|^2 \,d\bx
   d\bv \,+\, \nu \int_{\Omega\times\SSS^{d-1}} |\bq_h|^2 \,d\bx d\bv  
\\
  &&+\, \frac{1}{2}\,\sum_{\sigma_\bv\in\mE_\bv}\int_{\sigma_\bv} \left(\left| \bP_{\bv^\perp}\bv_{f_h}\cdot
    \bn_\bv\right|+\nu\,C_{11}\right)\, [f_h]_\bv^2\,+\, \nu\,C_{22} \, [\bq_h]_\bv^2 \,ds_\bv\,d\bx
\\
&&+\, \frac{1}{2}\,\sum_{\sigma_\bx\in\mE_\bx}\int_{\sigma_\bx}  |\bv\cdot \bn_x|\,[f_h]_\bx^2\,ds_\bx\,d\bv  \,=\,  -\frac{1}{2}\,\int_{\Omega\times\SSS^{d-1}} f_h^2
                                                      \, {\rm div}_\bv\left( \bP_{\bv^\perp}\bv_{f_h}\right) \,d\bx d\bv  .
\end{eqnarray*}
\end{proof}
 From this Lemma, we get $L^2$ boundedness
estimates on $(f_h,\bq_h)$ and on the
macroscopic quantities. 
\begin{proposition}
Under the assumptions of Lemma \ref{lem:stability}, consider the numerical
solution  $(f_h,\bq_h)\in\mG_h^k\times
\mU_h^k$ given by (\ref{eq:scheme:1a})-(\ref{eq:flux:2}) supplemented
with periodic boundary conditions. Then for any $t\geq 0$ 
\beq
\label{res1:prop:1}
\|f_h(t)\|^2_{L^2} 
\,+\, 2\,\nu \,e^{t}\, \int_0^t\|\bq_h(s)\|^2_{L^2}\, ds
\,\leq\, \|f_h(0)\|^2_{L^2} \,  e^{t} .
\eeq
 Furthermore $(\rho_h,\rho_h\bu_h)$ computed from the distribution
function $f_h$ satisfies for any $t\geq 0 $
\beq
\label{res2:prop:1}
\left\{\begin{array}{l}
\ds\|\rho_h(t)\|_{L^2(\Omega)} \,\leq \, {\rm vol}(\SSS^{d-1})^{1/2} \,
\|f_h(0)\|_{L^2}\, e^{t},
\\ \,\\ 
\ds\|\rho_h \bu_h(t)\|_{L^2(\Omega)}\,\leq \, {\rm vol}(\SSS^{d-1})^{1/2}  \,
\|f_h(0)\|_{L^2}\, e^{t}
\end{array}\right.
\eeq
and $\|\bu_h\|_{L^\infty(\Omega)} \leq 1$.
\label{prop:1}
\end{proposition}
\begin{proof}
Starting from Lemma \ref{lem:stability} and using the fact that
$\|\bv_{f_h}\|\leq 1$, we have
$$
\frac{d}{dt}\int_{\Omega\times\SSS^{d-1}} |f_h|^2 \,d\bx d\bv \,+\, 2\nu \int_{\Omega\times\SSS^{d-1}} |\bq_h|^2 \,d\bx d\bv
\,\leq\,  \int_{\Omega\times\SSS^{d-1}} f_h^2\,d\bx d\bv,
$$
or
$$
\frac{d}{dt}\left( e^{-t}\,\int_{\Omega\times\SSS^{d-1}} |f_h|^2 \,d\bx d\bv\right) \,+\, 2\nu \int_{\Omega\times\SSS^{d-1}} |\bq_h|^2 \,d\bx d\bv
\,\leq\,  0.
$$
Hence after integration, it yields to the $L^2$ estimate for any
$t\geq 0$,
$$
\int_{\Omega\times\SSS^{d-1}} |f_h(t)|^2 \,d\bx d\bv
\,+\, 2\,\nu\, e^{t}\, \int_0^t\int_{\Omega\times\SSS^{d-1}} |\bq_h(s)|^2
\,d\bx d\bv ds
\,\leq\,  e^{t} \int_{\Omega\times\SSS^{d-1}} |f_h(0)|^2 \,d\bx d\bv.
$$
The estimates on $\rho_h(t)$ and $\rho_h \bu_h(t)$ now easily follow since
$\bv\in\SSS^{d-1}$ and by application of the Cauchy-Schwarz
inequality
$$
\left\{\begin{array}{l}
\ds\|\rho_h(t)\|_{L^2(\Omega)} \,\leq \, {\rm vol}(\SSS^{d-1})^{1/2} \,
\|f_h(t)\|_{L^2},
\\ \,\\ 
\ds\|\rho_h \bu_h(t)\|_{L^2(\Omega)}\,\leq \, {\rm vol}(\SSS^{d-1})^{1/2}  \,
\|f_h(t)\|_{L^2}.
\end{array}\right.
$$
Hence, we conclude the proof of (\ref{res2:prop:1}) from (\ref{res1:prop:1}).
\end{proof}

\section{Proof of Theorem \ref{th:1}}
\setcounter{equation}{0}
\label{sec:4}

For any nonnegative integer $m$, $H^m(\Omega)$ denotes the
$L^2$-Sobolev space of order $m$ with the standard Sobolev norm and for $m=0$, we use $H^0(\Omega)=L^2(\Omega)$.

For any nonnegative integer $k$, let $\Pi_h$ be the $L^2$ projection
onto $\mG^k_h$, then we have the following classical result \cite{ciarlet}. 

\subsection{Basic results}
\begin{lemma}[{Approximation properties}]
\label{lem:appr}
There exists a constant $C>0$, such that  for any $g\in H^{m+1}(\Omega)$, the following hold: 
\beq
\label{ineg:interp}
\|g-\Pi^m g\|_{L^2(K)}\,+\,h_K^{1/2}\,\|g-\Pi^m g\|_{L^2(\partial K)}\,\leq\, C\;
h_K^{m+1}\,\|g\|_{H^{m+1}(K)},\quad \forall K\in\mT_h~,
\eeq
where the constant $C$ is independent of the mesh sizes $h_K$  but  depends on $m$ and the shape regularity parameters $\sigma_\bx$ and $\sigma_\bv$ of the mesh.
\end{lemma}
Moreover, we also remind the classical inverse inequality \cite{ciarlet}
\begin{lemma}[{Inverse inequality}]
\label{lem:inverse}
There exists a constant $C>0$, such that  for any $g\in P^m(K)$ or $P^m(K_\bx)\times P^m(K_\bv)$ with $K=(K_\bx\times K_\bv)\in\mT_h$,  the following holds: 
\begin{equation*}
\|\nabla_\bx g\|_{L^2(K)}\,\leq\, C\, h_{K_\bx}^{-1}\,\|g\|_{L^2(K)},\qquad \|\nabla_\bv g\|_{L^2(K)}\,\leq\, C \,h_{K_\bv}^{-1}\,\|g\|_{L^2(K)},
\end{equation*}
where the constant $C$ is independent of the mesh sizes $h_{K_\bx}$, $h_{K_\bv}$, but depends on $m$ and the shape regularity parameters $\sigma_\bx$ and $\sigma_\bv$ of the mesh.
\end{lemma}

Now let us start the error estimate analysis and consider
$(f,\bq=\nabla_\bv f)$ the exact solution to the kinetic equation
(\ref{kinetic:eq}) and  $(f_h,\bq_h)$ the approximated solution given
by (\ref{eq:scheme:1a})-(\ref{eq:flux:2}).

We introduce  the consistency error function $\delta_h$ and the projected error $\xi_{h}\in \mG_h^k\times \mU_h^k$  such that  
\beq
\label{err}
\left\{
\begin{array}{l}
\ds\xi_h\,=\,(\xi_{h,1}, \xi_{h,2}) \,=\, \left(\, f-\Pi_h f \,,\, \bq-\Pi_h\bq\,\right),
\\
\,
\\
\ds\delta_h\,=\,(\delta_{h,1}, \delta_{h,2}) \,=\, \left(\, \Pi_h f - f_h\,,\, \Pi_h \bq - \bq_h\,\right),
\end{array}\right.
\eeq
where $\Pi_h$ represents the $L^2$ projection onto $\mG_h^k$ and
$\mU_h^k$. Hence, we define the total error  $\veps_{h}=\delta_h +
\xi_h$.

We are now ready to prove the following Lemma 
\begin{lemma} [{Estimate of $\delta_h$}] 
Consider the numerical solution $(f_h,\bq_h)\in\mG_h^k\times \mU_h^k$ for
$k\geq 0$ given by \eqref{eq:scheme:1a}-\eqref{eq:flux:2} supplemented
with periodic boundary conditions. Then for any $h_0>0$, there exists a constant $C>0$
depending on $f$ and $h_0$, such that for $h\leq h_0$,
\beq
\label{deltah}
\frac{1}{2}\frac{d}{dt}\|\delta_{h,1}\|^2_{L^2}  \,+\, \nu \|\delta_{h,2}\|^2_{L^2} \,\leq\, C\,\left[ \|\delta_{h,1}\|_{L^2}^2 \,+\, h^{2k+1}\,+\,
      \|\bv_f-\bv_{f_h}\|_{L^\infty} \, \|\delta_{h,1}\|_{L^2}\right].
\eeq
\label{lmm:4.3}
\end{lemma}
\begin{proof}
On the one hand the numerical approximation $(f_h,\bq_h,\bv_{f_h})$ given by  \eqref{def:ah}-\eqref{def:bh} satisfies
\beq
\left\{
\begin{array}{l}
\ds a_h(f_h, \bq_h, \bv_{f_h}, g)\,=\,0, \quad \forall\, g\,\in\, \mG_h^k,
\\ \,\\
\ds b_h(f_h, \bq_h,\bu)\,=\,0, \quad \forall \,\bu\,\in\, \mU_h^k~.
\end{array}\right.
\label{err:1}
\eeq
On the other hand since the numerical fluxes of
\eqref{eq:scheme:1a}-\eqref{eq:flux:2} are consistent, the exact
solution $(f,\bq, \bv_f)$ satisfies
\beq
\label{err:2}
\left\{
\begin{array}{l}
\ds a_h(f, \bq, \bv_f, g)\,=\,0, \quad \forall\, g\,\in\, \mG_h^k,
\\ \,\\
\ds b_h(f, \bq,\bu)\,=\,0, \quad \forall \,\bu\,\in\, \mU_h^k~.
\end{array}\right.
\eeq
Then we notice that $\delta_{h}\in\mG_h^k\times \mU^k_h$; by taking
$g=\delta_{h,1}$ and $\bu=\delta_{h,2}$ in
\eqref{err:1} and \eqref{err:2} and subtracting the two equalities, one has
\begin{equation}
\left\{
\begin{array}{l}
\ds a_h(\delta_h,  \bv_{f_h},\delta_{h,1}) \,=\, -a_h(\xi_h, \bv_{f_h},
  \delta_{h,1}) \,+\, R_0,
\\ \,\\
\ds b_h(\delta_h, \delta_{h,2})\,+\, b_h(\xi_h, \delta_{h,2}) \,=\, 0,
\end{array}\right.
\label{err:3}
\end{equation}
where $R_0$ contains the nonlinear terms 
\begin{eqnarray*}
R_0 &:=& \int_{\Omega\times\SSS^{d-1}}f\, \bP_{\bv^\perp}
  \left(\bv_f-\bv_{f_h}\right)\,\cdot\nabla_\bv\delta_{h,1}\,d\bx d\bv
\\
&+& \sum_{\sigma_\bv\in \mE_\bv}\int_{\sigma_\bv} \widehat{f\,\bP_{\bv^\perp}(\bv_{f}-\bv_{f_h})} \cdot \bn_\bv \, [\delta_{h,1}]_\bv \,ds_\bv d\bx.
\end{eqnarray*}
Following the same lines as in the proof of Lemma \ref{lem:stability}
and using the definition of  $b_h$, we get
\begin{eqnarray*}
a_h(\delta_h, \bv_{f_h}, \delta_{h,1})&=&\frac{1}{2}\frac{d}{dt}\int_{\Omega\times\SSS^{d-1}} |\delta_{h,1}|^2 \,d\bx d\bv \,+\, \nu \int_{\Omega\times\SSS^{d-1}} |\delta_{h,2}|^2 \,d\bx d\bv
\\
&+& \frac{1}{2}\,\sum_{\sigma_\bv\in\mE_\bv}\int_{\sigma_\bv} \left(\left| \bP_{\bv^\perp}\bv_{f_h}\cdot
    \bn_\bv\right|+\nu\,C_{11}\right)\, [\delta_{h,1}]_\bv^2\,+\, \nu\,C_{22}
    \, [\delta_{h,2}]_\bv^2 \,ds_\bv\,d\bx
\\
&+&\frac{1}{2}\,\sum_{\sigma_\bx\in\mE_\bx}\int_{\sigma_\bx}
    |\bv\cdot \bn_x|\,[\delta_{h,1}]_\bx^2\,ds_\bx\,d\bv 
\\
&+&\frac{1}{2}\,\int_{\Omega\times\SSS^{d-1}} |\delta_{h,1}|^2
                                                      \, {\rm div}_\bv\left( \bP_{\bv^\perp}\bv_{f_h}\right) \,d\bx d\bv\,+\,\nu\,b_h(\xi_h,\delta_{h,2}).
\end{eqnarray*}
Therefore, using (\ref{err:3}), it yields 
\begin{eqnarray}
\label{e:00}
&&\frac{1}{2}\frac{d}{dt}\int_{\Omega\times\SSS^{d-1}} |\delta_{h,1}|^2 \,d\bx d\bv \,+\, \nu \int_{\Omega\times\SSS^{d-1}} |\delta_{h,2}|^2 \,d\bx d\bv
\\
\nonumber
&+& \frac{1}{2}\,\sum_{\sigma_\bv\in\mE_\bv}\int_{\sigma_\bv} \left(\left| \bP_{\bv^\perp}\bv_{f_h}\cdot
    \bn_\bv\right|+\nu\,C_{11}\right)\, [\delta_{h,1}]_\bv^2\,+\, \nu\,C_{22}
    \, [\delta_{h,2}]_\bv^2 \,ds_\bv\,d\bx
\\
\nonumber
&+&\frac{1}{2}\,\sum_{\sigma_\bx\in\mE_\bx}\int_{\sigma_\bx}
    |\bv\cdot \bn_x|\,[\delta_{h,1}]_\bx^2\,ds_\bx\,d\bv 
\\
\nonumber
&\leq&\frac{1}{2}\,\int_{\Omega\times\SSS^{d-1}} |\delta_{h,1}|^2
                                                      \left| {\rm
       div}_\bv\left( \bP_{\bv^\perp}\bv_{f_h}\right)\right| \,d\bx
       d\bv\,+\,\nu\,|b_h(\xi_h,\delta_{h,2})| + |R_0| + |a_h(\xi_h,\bv_{f_h},\delta_{h,1})|,
\end{eqnarray}
where we need to evaluate the right hand side.

On the one hand, we observe that $ \left|{\rm div}_\bv\left(
    \bP_{\bv^\perp}\bv_{f_h}\right)\right| \leq C$, so that the first
term is straightforward.

Then, we evaluate the second term $|b_h|$ in (\ref{e:00}). Using that $\xi_{h,2} = \nabla_\bv f - \Pi_h  \nabla_\bv f$ the
integral on the control volume vanishes and applying the Young
inequality and Lemma \ref{lem:appr}, we have
\begin{eqnarray}
\label{ii:1}
|b_h(\xi_h,\delta_{h,2})| &\leq&  \sum_{\sigma_\bv\in \mE_\bv}\int_{\sigma_\bv}\left|\widehat{ \xi_{h,1}} \,\bn_\bv \cdot [\delta_{h,2}]_\bv\right| \,d\bx\,ds_\bv,
\\
&\leq& \frac{C_{22}}{2}\,\sum_{\sigma_\bv\in \mE_\bv}\int_{\sigma_\bv}
       [\delta_{h,2}]_\bv^2\,d\bx\,ds_\bv \,+\, \frac{C}{2\, C_{22}}\,
       \|f\|_{H^{k+1}}^2 \,h^{2k+1}.
\nonumber
\end{eqnarray} 
Now we treat the third term $|R_0|$ in (\ref{e:00}) and use the fact that $f$ and $\bP_\bv^\perp$ are continuous and
$\bv_f-\bv_{f_h}$ does not depend on $\bv$, hence after an integration
by part and by consistency of the flux, all the integrals over
$\sigma_\bv\in\mE_\bv$ vanish and there exists a constant $C>0$, only
depending on $\|\nabla_\bv f\|_{W^{1,\infty}}$, such that
\beq
\label{r:0}
|R_0| \,\leq \, C\, \|\bv_f - \bv_{f_h} \|_{L^\infty} \, \| \delta_{h,2}\|_{L^2}.
\eeq  
Finally the last term in (\ref{e:00}) is $a_h(\xi_h, \bv_{f_h},
  \delta_{h,1})$, we split it in two parts   
\beq
\label{r1:r2}
a_h(\xi_h, \bv_{f_h}, \delta_{h,1}) \,=\, R_1 + R_2,
\eeq
where $R_1$ and $R_2$ are given by
$$
\left\{
\begin{array}{lll}
R_1 \,&=& \ds\int_{\Omega\times\SSS^{d-1}}\left(\frac{\partial \xi_{h,1}}{\partial t}\delta_{h,1}  - \xi_{h,1}\bv\cdot\nabla_\bx \delta_{h,1}
-\, \left( \bP_{\bv^\perp}\bv_{f_h}\,\xi_{h,1} - \nu\,\xi_{h,2}\right)\cdot\nabla_\bv \delta_{h,1}
\right)\,d\bx d\bv,
\\
\,
\\
R_2 \,&=& \ds-\,\sum_{\sigma_\bx\in \mE_\bx}\int_{\sigma_\bx}
        \widehat{\xi_{h,1} \,\bv}\cdot \bn_x\,
        [\delta_{h,1}]_\bx\,ds_\bx \,d\bv 
\\
\,
\\
\,& \,&\ds - \,\sum_{\sigma_\bv\in \mE_\bv}\int_{\sigma_\bv}
      \left(\widehat{\xi_{h,1}\,\bP_{\bv^\perp}\bv_{f_h}} \,-\,
      \nu\, \widehat{\xi_{h,2}} \right)  \cdot \bn_\bv \, [\delta_{h,1}]_\bv
      \,ds_\bv d\bx.
\end{array}
\right.
$$
Let us first  evaluate the term $R_1$ and decompose it as $R_1=R_{11}
+ R_{12} + R_{13}$, with
$$
|R_{11}| := \left| \int_{\Omega\times\SSS^{d-1}}\frac{\partial
      \xi_{h,1}}{\partial t}\delta_{h,1} \,d\bx d\bv \right| \,\leq\,
  C\, h^{k+1}\,
  \|\partial_t f\|_{H^{k+1}} \,  \|\delta_{h,1}\|_{L^2}
$$
and $R_{12}$ is
$$
R_{12} :=
\int_{\Omega\times\SSS^{d-1}}\xi_{h,1}(\bv-\bv_0)\cdot\nabla_\bx\delta_{h,1}
\,d\bx d\bv +  \int_{\Omega\times\SSS^{d-1}}\xi_{h,1}\bv_0\cdot\nabla_\bx\delta_{h,1}
\,d\bx d\bv,  
$$
where $\bv_0$ be the $L^2$ projection of the function $\bv$ onto
the piecewise constant space with respect to $\mT_h$. Hence , by
definition of the $L^2$ projection $\xi_{h,1}=f-\Pi_h f $ the last
term vanishes and we have from Lemma \ref{lem:appr} and since $\|\bv-\bv_0\|_{L^\infty}\leq
C \,h$,
$$
|R_{12}| \,\leq \, C\, \|\bv-\bv_0\|_{L^\infty} \, h^{k} \,
\|f\|_{H^{k+1}}\,\|\delta_{h,1} \|_{L^2} \, \leq \, C \, h^{k+1} \,
\|f\|_{H^{k+1}}\,\|\delta_{h,1} \|_{L^2}.
$$
Finally, we evaluate $R_{13}$ defined as
\begin{eqnarray*}
R_{13} &:=& -\int_{\Omega\times\SSS^{d-1}}
            \bP_{\bv^\perp}(\bv_{f_h}-\bv_f)
            \,\xi_{h,1}\cdot\nabla_\bv \delta_{h,1}\,d\bx d\bv  
\\
&-& \int_{\Omega\times\SSS^{d-1}}
    \bP_{\bv^\perp}\bv_f\,\xi_{h,1}\cdot\nabla_\bv \delta_{h,1}\,d\bx
    d\bv 
\\
&+& \nu\, \int_{\Omega\times\SSS^{d-1}} \xi_{h,2}\cdot\nabla_\bv \delta_{h,1}\,d\bx d\bv.
\end{eqnarray*}
Using the definition of the $L^2$ projection $\xi_{h,2}=\bq-\Pi_h \bq
$,  we first observe  that the last term vanishes
$$
\int_{\Omega\times\SSS^{d-1}}\xi_{h,2}\cdot\nabla_\bv
\delta_{h,1}\,d\bx d\bv = 0,
$$ 
then we proceed as on the estimate of $R_{12}$ by introducing the
$L^2$ projection of the function $\bP_{\bv^\perp}\,\bv_f$ onto the
piecewise constant space with respect to $\mT^h$, whereas we apply the
Cauchy-Schwarz inequality to treat the first term. It yields that there exists a constant $C>0$,
$$
|R_{13}| \,\leq\, C\,\left( \|\bv_f -
  \bv_{f_h}\|_{L^\infty}\,\|\xi_{h,1}\|_{L^2}\, \|\nabla_\bv
  \delta_{h,1}\|_{L^2} \,+\, h^{k+1} \, \|f\|_{H^{k+1}}\,\|\delta_{h,1}\|_{L^2} \right).
$$
Again we apply  the two Lemmas \ref{lem:inverse} and \ref{lem:appr},
which gives that 
$$
|R_{13}| \,\leq\, C\, h^{k} \, \left( \|\bv_f -
  \bv_{f_h}\|_{L^\infty}  \,+\, h\right)  \, \|f\|_{H^{k+1}}\,\|\delta_{h,1}\|_{L^2}.
$$
Gathering these results on $R_{11}$, $R_{12}$ and $R_{13}$, we get the
following estimate on $R_1$,
\beq
\label{r:1}
|R_1| \,\leq\, C\, h^{k}\,\left( \|\bv_f - \bv_{f_h} \|_{L^\infty}
  \,+\, h\right)  \, \|f\|_{H^{k+1}}\,\| \delta_{h,1}\|_{L^2}.
\eeq
Now we want to estimate the term $R_2$ containing the fluxes such that
$R_2 = R_{21} +R_{22} +R_{23} $, with 
\begin{eqnarray*}
|R_{21}| &:=& \left|\sum_{\sigma_\bx\in \mE_\bx}\int_{\sigma_\bx}
        \widehat{\xi_{h,1} \,\bv}\cdot \bn_x\,
        [\delta_{h,1}]_\bx\,ds_\bx \,d\bv \right| \\
&\leq & \frac{1}{2}\,\sum_{\sigma_\bx\in \mE_\bx}\int_{\sigma_\bx}
       |\bv\cdot \bn_x|\, [\delta_{h,1}]_\bx^2\,ds_\bx \,d\bv \,+\,
        \frac{1}{2}\,\sum_{\sigma_\bx\in \mE_\bx}\int_{\sigma_\bx}
        |\xi_{h,1}|^2\,ds_\bx \,d\bv
\\
&\leq& \frac{1}{2}\,\sum_{\sigma_\bx\in \mE_\bx}\int_{\sigma_\bx}
       |\bv\cdot \bn_x|\, [\delta_{h,1}]_\bx^2\,ds_\bx \,d\bv \,+\,  C\, \|f\|_{H^{k+1}}^2\, h^{2k+1},
\end{eqnarray*}
whereas from Lemma \ref{lem:appr}
\begin{eqnarray*}
|R_{22}| &:=& \left|\sum_{\sigma_\bv\in \mE_\bv}\int_{\sigma_\bv}
      \widehat{\xi_{h,1}\,\bP_{\bv^\perp}\bv_{f_h}} \cdot \bn_\bv \, [\delta_{h,1}]_\bv
      \,ds_\bv d\bx\right|,
\\
&\leq& \frac{1}{2}\sum_{\sigma_\bv\in \mE_\bv}\int_{\sigma_\bv}
      \left|\bP_{\bv^\perp}\bv_{f_h} \cdot \bn_\bv\right|  \,
     [\delta_{h,1}]_\bv^2 
      \,ds_\bv d\bx\, + \, \frac{1}{2}\sum_{\sigma_\bv\in \mE_\bv}\int_{\sigma_\bv}
      \left|\bP_{\bv^\perp}\bv_{f_h} \cdot \bn_\bv\right| \, |\xi_{h,1}|^2
      \,ds_\bv d\bx,
\\
&\leq& \frac{1}{2}\sum_{\sigma_\bv\in \mE_\bv}\int_{\sigma_\bv}
      \left|\bP_{\bv^\perp}\bv_{f_h} \cdot \bn_\bv\right|  \,
     [\delta_{h,1}]_\bv^2 
      \,ds_\bv d\bx\, +\, C\, \|f\|_{H^{k+1}}^2\, h^{2k+1}.
\end{eqnarray*}
Finally, from the Young inequality and Lemma \ref{lem:appr}, we get that for any $\eta>0$,
\begin{eqnarray*}
|R_{23}| &:=& \nu\,\left|\sum_{\sigma_\bv\in \mE_\bv}\int_{\sigma_\bv}
              \widehat{\xi_{h,2}} \cdot \bn_\bv \, [\delta_{h,1}]_\bv
              \,ds_\bv d\bx\right|,
\\
&\leq &\frac{\nu\eta}{2}\sum_{\sigma_\bv\in
        \mE_\bv}\int_{\sigma_\bv}[\delta_{h,1}]_\bv^2\,ds_\bv
        d\bx\,+\, \frac{C\,\nu}{2\eta} \, \|f\|_{H^{k+1}}^2\, h^{2k+1}.
\end{eqnarray*}
Gathering these results on $R_{21}$, $R_{22}$ and $R_{23}$, we get the following estimate on
$R_2$ with $\eta=C_{11}$,
\begin{eqnarray}
\label{r:2}
|R_2| &\leq& \frac{1}{2}\,\sum_{\sigma_\bx\in \mE_\bx}\int_{\sigma_\bx}
       |\bv\cdot \bn_x|\, [\delta_{h,1}]_\bx^2\,ds_\bx \,d\bv 
\\ &+&  \frac{1}{2}\sum_{\sigma_\bv\in \mE_\bv}\int_{\sigma_\bv}
      \left|\bP_{\bv^\perp}\bv_{f_h} \cdot \bn_\bv\right|  \,
     [\delta_{h,1}]_\bv^2 
      \,ds_\bv d\bx
\nonumber
\\ &+& \frac{C_{11}\,\nu}{2}\sum_{\sigma_\bv\in
        \mE_\bv}\int_{\sigma_\bv}[\delta_{h,1}]_\bv^2\,ds_\bv d\bx \;+\, C\,\|f\|_{H^{k+1}}^2\, h^{2k+1}.
\nonumber
\end{eqnarray}
To conclude the proof, we consider  again (\ref{e:00})
and use the estimates obtained in (\ref{ii:1})- (\ref{r:2}),  it yields
\begin{eqnarray*}
&&\frac{1}{2}\frac{d}{dt}\int_{\Omega\times\SSS^{d-1}} |\delta_{h,1}|^2 \,d\bx d\bv \,+\, \nu \int_{\Omega\times\SSS^{d-1}} |\delta_{h,2}|^2 \,d\bx d\bv
\\
&\leq&C\,\left( \|\delta_{h,1}\|_{L^2}^2 \,+\, h^{2k+1}\,+\,
       \left( (1+h^k)\|\bv_f-\bv_{f_h}\|_{L^\infty} \,+\, h^{k+1}\right) \|\delta_{h,1}\|_{L^2}\right).
\end{eqnarray*}
Finally for any $h_0>0$, it yields to the result (\ref{deltah}) for  $h\leq h_0$.
\end{proof}
Now to complete the proof of convergence it remains to estimate the
error on the velocity field $\bv_{f_h}$.

\begin{lemma} [{Estimate of $\bv_{f_h}$}] 
Consider the numerical solution $(f_h,\bq_h)\in\mG_h^k\times \mU_h^k$ for
$k\geq 0$ given by \eqref{eq:scheme:1a}-\eqref{eq:flux:2} supplemented
with periodic boundary conditions. Then there exists a constant $C>0$
such that
\beq
\label{vfhvf}
\|\bv_f-\bv_{f_h}\|_{L^\infty}  \leq C\,\|\veps_{h,1}\|_{L^2}.
\eeq
\label{lmm:4.4}
\end{lemma}
\begin{proof}
Thanks to (\ref{res2:prop:1})  in Proposition \ref{prop:1}, there
exists a constant $C>0$, only depending on the dimension $d$, $\Omega$ and the support of
$k$ and $\phi$ such that  
$$
\|\bJ_f - \bJ_h \|_{L^\infty}  +
\|\bR_f - \bR_h \|_{L^\infty} \,\leq\, C\left( \|\phi^\prime\|_{L^2} \,+\, \|k\|_{L^2}
\right) \, \|\veps_{1}\|_{L^2(\Omega\times\SSS^{d-1})}.
$$
Then we evaluate the error $\bv_f - \bv_{f_h}$ by
$$
| \bv_f - \bv_{f_h} |=     \left|\frac{(\bJ_f+\bR_f) -
  (\bJ_h+\bR_h)}{|\bJ_f+\bR_f|} \,+\,  (\bJ_h+\bR_h) \,
\frac{|\bJ_h+\bR_h| -|\bJ_f+\bR_f|}{|\bJ_h+\bR_h|\,|\bJ_f+\bR_f|} \right|.
$$
Hence from assumption (\ref{hyp:01}) on $\bJ_f$ and $\bR_f$, there
exists a new constant $C_f>0$ depending on the exact solution $f$,
$k$, $\phi$ and $\Omega$  such
that 
$$
\| \bv_f - \bv_{f_h} \|_{L^\infty} \,\leq \frac{2}{\xi_T}\, \left( \|\bJ_f-\bJ_h\|_{L^\infty}  +
  \|\bR_f -\bR_h\|_{L^\infty} \right) \leq C_f\, \|\veps_{h,1}\|_{L^2}.
$$
\end{proof}

\subsection{Error estimates $\|f-f_h\|_{L^2}$}
To prove  Theorem \ref{th:1}, we first obtain the $L^2$
estimate on $(f_h,\bq_h)$, which  is a direct consequence of the
stability estimate proven in Lemma \ref{lem:stability} :
there exists a constant $C_T>0$ such that for any $t\in [0,T]$
$$
\|f_h(t)\|_{L^2} \,+\, \|\bq_h(t)\|_{L^2}\,\leq C_T.
$$
The error estimate follows by applying Lemma \ref{lmm:4.3} with the
estimate on $\|\bv_f-\bv_{f_h}\|_{L^\infty}$ given in   Lemma
\ref{lmm:4.4}, it yields
$$
\frac{1}{2}\frac{d}{dt}\|\delta_{h,1}\|^2_{L^2}  \,+\, \nu \|\delta_{h,2}\|^2_{L^2} \,\leq\, C\,\left[ \|\delta_{h,1}\|_{L^2}^2 \,+\, h^{2k+1}\,+\,
      \|\veps_{h,1}\|_{L^2} \, \|\delta_{h,1}\|_{L^2}\right].
$$
Then we remind that $\veps_h=\delta_h+ \xi_h$, where $\xi_{h,1}=f -
\Pi_h f$ satisfies from Lemma \ref{lem:appr}
\beq
\|\xi_{h,1}\|_{L^2}=\|f-\Pi^m f\|_{L^2(K)}\,\leq\, C\;
h_K^{k+1}\,\|f\|_{H^{k+1}(K)},\quad \forall K\in\mT_h~,
\label{dd:1}
\eeq
hence for any $h_0>0$, there exists another constant $C>0$, depending
on $f$ and $h_0$, such that $h\leq h_0$ and
$$
\frac{1}{2}\frac{d}{dt}\|\delta_{h,1}\|^2_{L^2}  \,+\, \nu \|\delta_{h,2}\|^2_{L^2} \,\leq\, C\,\left[ \,\|\delta_{h,1}\|_{L^2}^2 \,+\, h^{2k+1}\,\right].
$$
Applying the Gronwall's Lemma, we get that there exists a  constant $C_T>0$, depending
on $f$, $T$ and $h_0$, such that $h\leq h_0$ and for all $t\in [0,T]$
\beq
\label{dd:2}
\|\delta_{h,1}(t)\|_{L^2}  \,+\, \left(\int_0^t\|\delta_{h,2}\|_{L^2}^2 ds\right)^{1/2} \,\leq\, C_T\, h^{k+1/2}.
\eeq
Finally gathering (\ref{dd:1}) and (\ref{dd:2}) and using the same
kind of estimate as (\ref{dd:1}) for $\nabla_\bv f - \Pi_h \nabla_\bv f$, we get that for  $h\leq h_0$ and for all $t\in [0,T]$,
$$
\|f(t) - f_h(t)\|_{L^2}  \,+\, \left(\int_0^t\|\nabla_\bv f(s) -
  \bq_h(s)\|_{L^2}^2 ds \right)^{1/2} \,\leq\, C_T\, h^{k+1/2}.
$$ 
\section{Numerical simulations}
\setcounter{equation}{0}
\label{sec:5}
We now present several numerical experiments and simply choose
$C_{11}=C_{22}=1$. We first propose an accuracy test to verify the
order of accuracy of the method and then give two examples on creation
of vortices and band formation.

\subsection{Accuracy test}
We first consider the model  (\ref{kinetic:eq}), where the velocity $\bv_f$ is
fixed and given by $\bv_f = \bx \, t$.  The initial datum is
$$
f_0(\bx,\bv) = \frac{1}{2\pi \,\nu}
\,\exp\left(-\frac{\|\bx\|^2}{2\nu}\right), \quad\bx \in\Omega,
$$
where the computational domain is chosen as $\Omega= [-1,1]^2$. Hence the exact solution is given by 
$$
f(t,\bx,\bv) = \frac{1}{2\pi \,\nu} \,\exp\left(-\frac{\|\bx-\bv t\|^2}{2\nu}\right).
$$
In the numerical simulations, uniform meshes are used, with $N$ cells in each direction. In addition, the
third order TVD Runge–Kutta method is applied in time, with the CFL
number for the upwind and  alternating flux in $P_1$ and $P_2$
cases. In Tables \ref{tab1} and \ref{tab2}, we present the error
$\varepsilon_{N}^1$ (resp. $\varepsilon_{N}^\infty$) on the exact solution
for $L^1$ (resp. $L^\infty$) norm for $k=1$ and $2$ with
$$
\varepsilon_N^1(t) \,=\, \int_{\Omega\times \SSS^1}  | f(t) - f_h(t)|
d\bx \, d\bv, \quad  \varepsilon_N^\infty(t) \,=\,
\sup_{(\bx,\bv)\in\Omega\times\SSS^1} |f(t) - f_h(t)|.
$$

 We observe that the schemes with the
upwind and alternating fluxes achieve optimal $(k + 1)$-th order accuracy in approximating the solution compared to $(k + 1/2)$-th
order of accuracy established in the previous section.

\begin{table}[H]	
$$
\begin{tabular}{|c|c|c|c|c|c|c}
			\hline
			 \,& $N$  & $L^1$ error &  order & $L^\infty$  error & order \\ \hline
                        \, & 16 & 3.09753e-00 & --     & 1.74318e-00  & --  \\ 
                        \, & 24 & 1.57566e-00 &  1.7   & 9.76815e-01  &  1.45 \\ 
                    $k=1$ & 32 & 9.21703e-01 &  1.7   & 5.97741e-01  &  1.54\\  
                        \, & 48 & 4.03124e-01 & 1.95  & 2.73951e-01   & 1.83\\ 
                        \, & 64 & 2.18455e-01 & 2.00  & 1.51392e-01   & 1.98\\ \hline
\end{tabular}
$$

\caption{{\bf Accuracy test.}  Error norm $\varepsilon^1_N$ and
$\varepsilon^\infty_N$ for $k=1$ where $N$ represents the number of
points in each direction. }
\label{tab1}
\end{table}

\begin{table}[H]
$$	
\begin{tabular}{|c|c|c|c|c|c|c|c}
			\hline
                    \,& $N$  & $L^1$ error &  order & $L^\infty$  error & order \\ \hline
			\,& 16 & 8.61814e-01 &  --   & 5.67255e-01  &  --   \\ 
			\,& 24 & 2.71516e-01 & 2.85 & 1.95118e-01 & 2.63  \\ 
			$k=2$ & 32 & 1.13208e-01 & 2.93 & 8.55585e-02 & 2.73   \\ 
			& 48 & 3.18561e-02 & 3.09 & 2.53507e-02 & 2.94  \\ 
			& 64 & 1.41038e-02 & 3.00 & 1.05810e-02 & 3.01  \\ \hline
\end{tabular}
$$
\caption{{\bf Accuracy test.} Error norm $\varepsilon^1_N$ and
$\varepsilon^\infty_N$ for $k=2$ where $N$ represents the number of
points in each direction.  }
\label{tab2}
\end{table}

\subsection{Taylor-Green vortex problem}
We  now consider the model  (\ref{kinetic:eq})-(\ref{vel:eq}) with
periodic boundary conditions in $\Omega=(0,10)^2$, where
the velocity $\bv_f$ is given as $\bR_f\equiv 0$ and
$$
\bJ_f(t,\bx)\, =\, \int_{\Omega\times\SSS^{d-1}} k(|\bx - \bx'|) \bv' \,
 f(t,\bx',\bv') \, d\bx'd\bv', 
$$
with $k(r)=\exp(-r^2/(2\sigma^2))$ and  $\sigma=0.1$. Here we neglect
the repulsion force $\bR_f$ and only take into account the alignment
of particles with  averaged velocity $\bv_f$.

We compare the numerical solutions provided by the local discontinuous
Galerkin method with the one obtained with the particle method in
\cite{Dimarco, Dimarco2}.
 The initial data are
$$
f_0(\bx,\bv) = \rho_0\, \left(  2\,+\, v_x\,\Omega_x(\bx) \,+\, v_y\,\Omega_y(\bx)\right),
$$
where $\bv=(\cos\theta,\sin\theta)$,  and 
$$
\left\{
\begin{array}{l}
\ds\Omega_x \,=\,
  +\frac{1}{3}\left[\sin\left(\frac{\pi\,x}{5}\right)\cos\left(\frac{\pi\,y}{5}\right)
  \,+\, \sin\left(\frac{3\pi\,x}{10}\right)\cos\left(\frac{3\pi\,y}{10}\right)
  \,+\, \sin\left(\frac{\pi\,x}{2}\right)\cos\left(\frac{\pi\,y}{2}\right)
  \right]
\\
\,
\\
\ds\Omega_y \,=\,
  -\frac{1}{3}\left[\cos\left(\frac{\pi\,x}{5}\right)\sin\left(\frac{\pi\,y}{5}\right)
  \,+\, \cos\left(\frac{3\pi\,x}{10}\right)\sin\left(\frac{3\pi\,y}{10}\right)
  \,+\, \cos\left(\frac{\pi\,x}{2}\right)\sin\left(\frac{\pi\,y}{2}\right)
  \right] 
\end{array}\right.
$$
with $\bx=(x,y)\in (0,10)^2$.

This model is supplemented by periodic boundary conditions in both directions. The numerical
parameters for the kinetic model (\ref{kinetic:eq})-(\ref{vel:eq})
are : $\Delta x = \Delta y = 0.2$, $\Delta t = 0.01$. In Figure
\ref{fig:1}, we report the density $\rho$ and the flux
direction $\bU$ at different time $t\in (0,30)$ given by
$$
\rho(t,\bx) = \int_{\SSS^1} f(t,\bx,\bv)\,d\bv, \quad \rho\,\bU(t,\bx) = \int_{\SSS^1} \bv\,f(t,\bx,\bv)\,d\bv.
$$
We find a very good agreement with the results in \cite{Dimarco, Dimarco2} for
agent based models (\ref{eq:particle_1}) and macroscopic models  in
spite of the quite complex structure of the solution (see Figures 4 and
5 in \cite{Dimarco} for short time $t=5$). In our simulation, we
also present simulations for large time and observe the time evolution
of vortices.

\begin{center}
\begin{figure}[!ht]
 \begin{tabular}{cc}
\includegraphics[width=7.75cm]{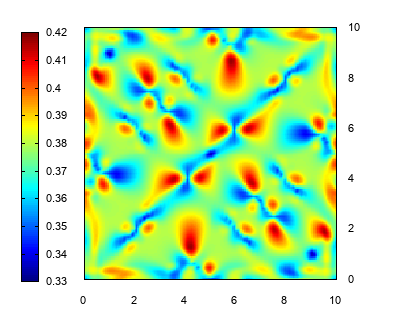} &    
\includegraphics[width=7.75cm]{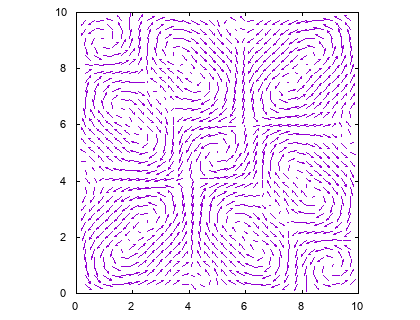} 
\\
\includegraphics[width=7.75cm]{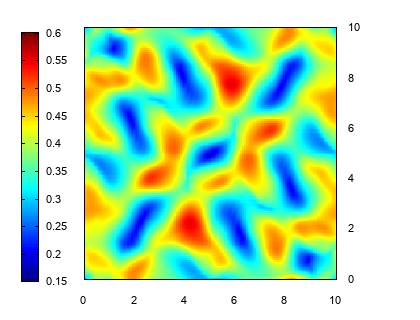} &    
\includegraphics[width=7.75cm]{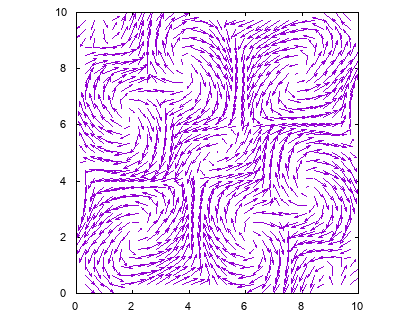} 
\\
\includegraphics[width=7.75cm]{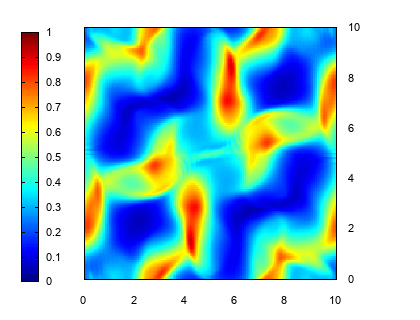} &    
\includegraphics[width=7.75cm]{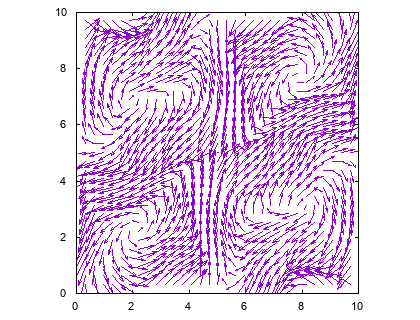} 
\\
(a)  & (b)  
\end{tabular}
\caption{ {\bf Taylor-Green vortex problem.} Numerical solution obtained with a
(a) density $\rho$, (b) mean velocity $\bu$ at time $t=5$, $t=15$  and $t=30$.}
 \label{fig:1}
\end{figure}
\end{center}

Finally we also propose  in Figure \ref{fig:2}  the time evolution of the local averaged
velocity $\bv_f$ given in (\ref{vel:eq}) and the persistence of
several vortices for large time $t\simeq 30$.
 
\begin{center}
\begin{figure}[!ht]
 \begin{tabular}{cc}
\includegraphics[width=7.75cm]{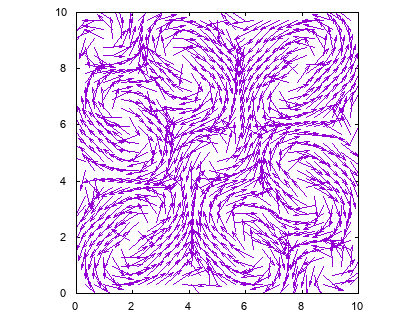} &
\includegraphics[width=7.75cm]{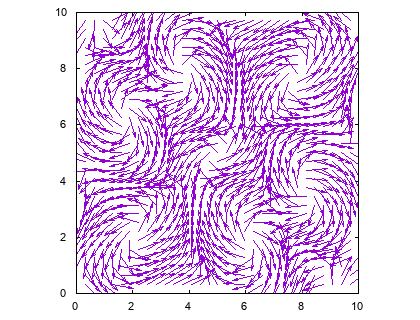} 
\\
$t= 5$ & $t=15$
\\
\includegraphics[width=7.75cm]{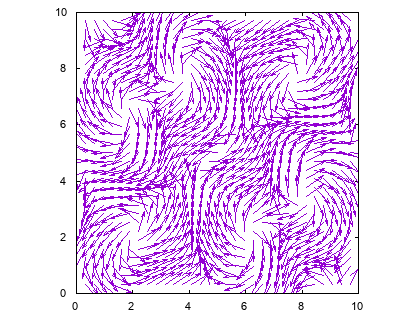} &
\includegraphics[width=7.75cm]{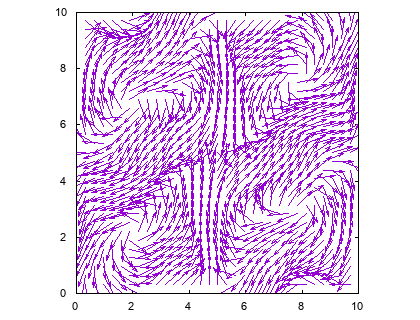} 
\\
$t=30$  & $t=50$  
\end{tabular}
\caption{ {\bf Taylor-Green vortex problem.} Numerical solution $\bv_f$ at time $t=5$, $t=15$, $t=20$ and $t=30$.}
 \label{fig:2}
\end{figure}
\end{center}

\subsection{Formation of bands problem}
We  still consider the kinetic model (\ref{kinetic:eq})-(\ref{vel:eq})
but with a different scaling for $\veps>0$, 
$$
\frac{\partial f}{\partial t} \,+\, \bv \cdot \nabla_\bx f = - \frac{1}{\veps}{\rm div}_\bv\left[ \bP_{\bv^\bot} \bv_f\, f - \nu \,\nabla_\bv f\right],
$$
where $\bv_f$ is defined as previously.  We set periodic boundary
conditions in $\Omega=(-1/2,1/2)\times (0,1)$, and the initial data is
given by
$$
f_0(\bx,\theta) = \left(1+\frac{1}{2}\,\cos(\theta)\right)\,
\left(1+\frac{3}{5}\,\sin(2\pi\,x) + \frac{3}{10}\,\cos(2\pi\,y)
\right), \quad \bx=(x,y)\in\Omega,\,\,\theta\in (0,2\pi). 
$$
We choose $\veps=0.25$ and $\nu=0.005$ and we investigate the long
time behavior of the numerical solution. On the one hand, we report
the time evolution of the density $\rho$ in Figure \ref{fig:2.1} and
observe after time $t$ larger than $15$, the formation of a band which
propagates with an horizontal velocity of speed $\simeq$ 1. Such a
behaviour has been already observed for numerical simulations of
stochastic models with only local alignment interactions \cite{chate}
as (\ref{eq:particle_1}). These moving structures appear for large
enough systems after some transient. Then, they extend transversally
with respect to the mean direction of motion. The advantage of kinetic
models as (\ref{kinetic:eq})-(\ref{vel:eq})  is that bands can be
described quantitatively through local quantities, such as the local
density $\rho$, but also the mena velocity $\bu$ and the local
averaged mean velocity $\bv_f$.

\begin{center}
\begin{figure}[!ht]
 \begin{tabular}{cc}
\includegraphics[width=7.75cm]{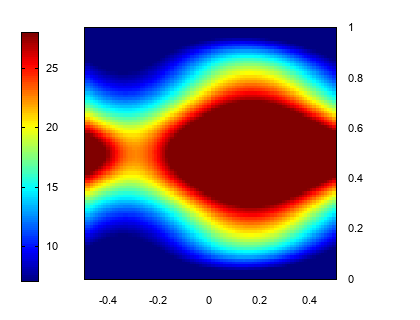} &    
\includegraphics[width=7.75cm]{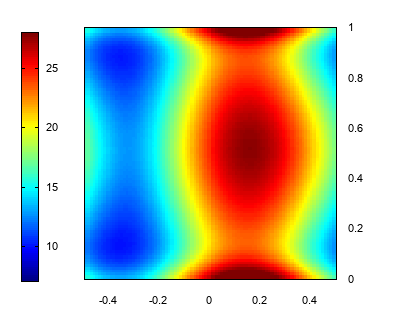} 
\\
\includegraphics[width=7.75cm]{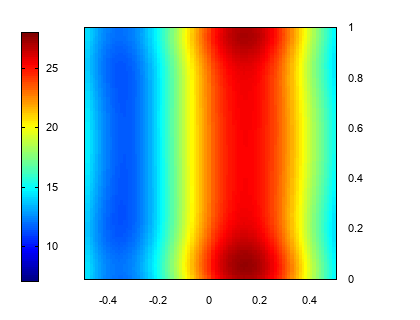} &    
\includegraphics[width=7.75cm]{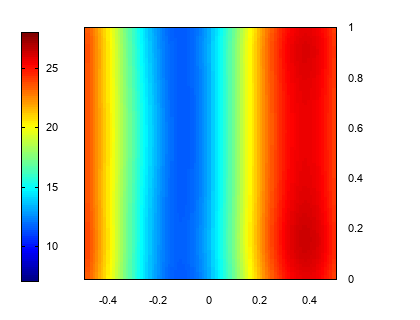} 
\\
\includegraphics[width=7.75cm]{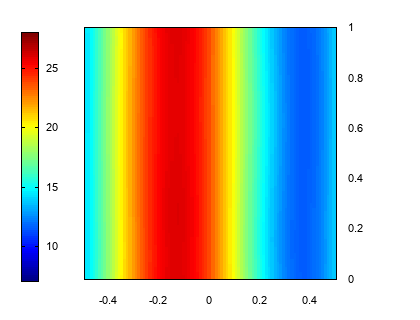} &    
\includegraphics[width=7.75cm]{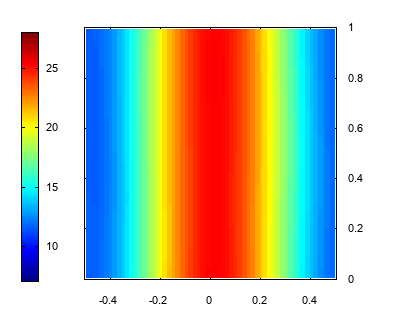} 
\end{tabular}
\caption{ {\bf Formation of bands problem.} Numerical solution of the density $\rho$ at time $t=2$, $t=9$,  $t=13$,  $t=17.25$, $t=22.75$   and $t=30$.}
 \label{fig:2.1}
\end{figure}
\end{center}

We finally propose  in Figure \ref{fig:2.2}, a snapshot of  the mean velocity
$\bu$ and the local averaged
velocity $\bv_f$ given in (\ref{vel:eq}) at the final time $t=30$. 
\begin{center}
\begin{figure}[!ht]
 \begin{tabular}{cc}
\includegraphics[width=7.75cm]{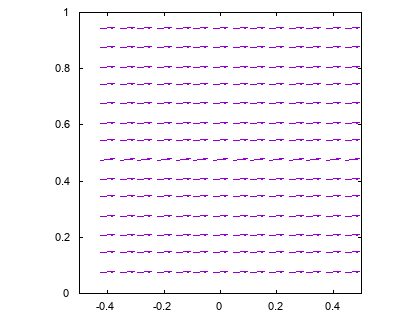} &
\includegraphics[width=7.75cm]{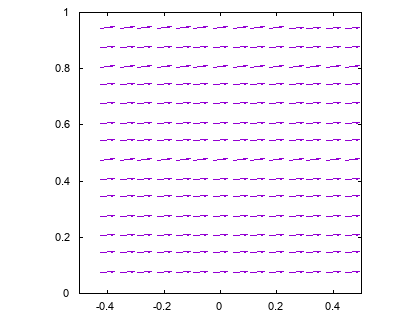} 
\\
(a) $\bu(t=30)$ & (b) $\bv_f(t=30)$  
\end{tabular}
\caption{ {\bf Formation of bands problem.} Numerical solution (a) $\bu$
  and (b) $\bv_f$ at time $t=30$.}
 \label{fig:2.2}
\end{figure}
\end{center}

\section{Conclusion and perspective}
\setcounter{equation}{0}
\label{sec:6}
In this paper we proposed a discontinuous Galerkin discretization
technique  for a kinetic model of self-alignment introduced in \cite{DM,DLMP}.  The main feature of
this approach is to guarantee the accuracy and stability for the $L^2$
norm. Furthermore,  we performed a complete analysis to get  error
estimates for smooth solutions. The
scheme has been tested using an exact solution where the order of
accuracy  has been verified. The proposed method has been applied to
study  the long time  dynamics of this system and can be further
investigated to improve the model.

\section*{Acknowledgements}
Francis Filbet acknowledges the
Division of Applied Mathematics, Brown University for the invitation
in January/February 2016, where the present work has been initiated.
The research of Chi-Wang Shu is partially supported by
DOE grant DE-FG02-08ER25863 and NSF grant DMS-1418750.

\bibliographystyle{plain}

\begin{flushleft} \signFF \end{flushleft}
\begin{flushright} \signCWS
\end{flushright}

\end{document}